\documentclass[letterpaper,10pt,conference]{IEEEtran}

\usepackage{amsthm}
\usepackage{amsmath}
\usepackage{amssymb,bbm,mathtools,amsbsy}
\usepackage{empheq}
\usepackage{eucal}
\usepackage{mathrsfs}
\usepackage{tikz}
\usepackage{float}
\usepackage{algorithm,algorithmic}
\usepackage{caption}
\usepackage{tikz}
\usetikzlibrary{arrows.meta}

\usepackage{subcaption}
\usepackage{balance}

\usepackage{siunitx}

\allowdisplaybreaks

\makeatletter
\let\NAT@parse\undefined
\makeatother
\usepackage{hyperref}

\newcommand*\Bell{\ensuremath{\boldsymbol\ell}}

\newcommand{\mcal}[1]{\mathcal{#1}}

\newcommand{\ldef}{\stackrel{\Delta}{=}}

\newcommand{\mb}[1]{\mathbf{#1}}

\newcommand{\tth}[1]{{#1}^{\text{th}}}

\newcommand{\norm}[1]{\left\lVert #1 \right\rVert}
\newcommand{\bm}[1]{
	\begin{bmatrix}
		#1
	\end{bmatrix}
}

\newcommand{\onevec}[1]{\mathbf{1}_{#1}}
\newcommand{\real}[1]{\mathbb{R}^{#1}}

\newtheorem{thm}{Theorem}
\newtheorem{lemma}{Lemma}

\newtheorem{assmp}{Assumption}
\newtheorem{remark}{Remark}

\title{On LinDistFlow Model Congestion Pricing:\\Bounding the Changes in Power Tariffs}
\author{
\IEEEauthorblockN{Shourya Bose, Kejun Chen, and Yu Zhang} 
\IEEEauthorblockA{ECE Department, University of California, Santa Cruz}
\IEEEauthorblockA{Emails: \texttt{\{shbose,\,kchen158,\,zhangy\}@ucsc.edu}
}}%

\begin{document}
	
\maketitle

\begin{abstract}
The optimal power flow (OPF) problem is an important  mathematical program that aims at obtaining the best operating point of an electric power grid. The optimization problem typically minimizes the total generation cost subject to certain physical constraints of the system. The so-called linearized distribution flow (LinDistFlow) model leverages a set of linear equations to approximate the nonlinear AC power flows.
In this paper, we consider an OPF problem based on the LinDistFlow model for a single-phase radial power network. We derive closed-form solutions to the marginal values of both real and reactive power demands. We also derive upper bounds on the congestion price (a.k.a. `shadow price'), which denotes the change in marginal demand prices when the apparent power flow limits of certain lines are binding at optimum. Various cases of our result are discussed while simulations are carried out on a $141$-bus radial power network.
\end{abstract}

\section{Introduction}
Radially energized power networks are prevalent in grid-scale power systems such as utility distribution networks and microgrids~\cite{AAS-OPM:2018}. They are defined as power distribution networks wherein the energized section assumes a topology of a \emph{connected tree}. Traditionally, the only generation source in a radial network is a sub-station as the upstream to the network, which is interfaced with some high-voltage transmission network using power conversion devices.  
	With the advent of consumer-level generation devices such as photovoltaic panels, wind turbines, and microturbines, radial networks can now accommodate \emph{prosumers}, i.e. agents on the network which can either inject/withdraw power in/from the network. Desirable set points of generation power can be determined by solving for set points which are optimal with respect to some generation cost function, subject to physical and operational constraints. This optimization problem, well known as \emph{Optimal Power Flow} (OPF) was first introduced in literature by Carpentier~\cite{JC:1979,MH-FDG:1991}. Depending on the nature of constraints and the cost function, OPF can be variously categorized as DC-OPF, AC-OPF, security constrained OPF, etc~\cite{ZY-HZ-QX-CK:2018}. The former two are different ways of modeling physics of the network, while the latter adds security or contingency constraints meant to ensure robust operation of the network. It is important to note that in this article we only consider OPF problems wherein the network operator and prosumers seek to optimize the same cost function. This is as opposed to scenarios wherein the prosumers may seek to optimize a cost function different from that of the network operator~\cite{WT-RJ:2013,SR-ALL-YC:2020}.
	
	The DC-OPF, which models the physics of the network using a set of linear equations, has been widely researched and used in practice for transmission networks, where the topology may be meshed. Combined with administrative constraints, DC-OPF provides reasonably accurate set points \emph{vis-a-vis} the optimal AC-OPF solution~\cite{KB:2021}. However, the DC-OPF linearization does not model reactive power injection, which poses a problem in analyzing radial distribution networks with devices such as PVs and WTs having controllable inverters. To counter this drawback, there has been significant recent research on the \emph{LinDistFlow} equation~\cite{JH-BC-XZ-AB:2021}. First introduced by Baran and Wu~\cite{MEB-FFW:1989}, this is a set of linearized equations which describes the physics of the network (possibly multi-phase with high R/X ratios). The underlying condition of LinDistFlow is that there are no power line losses, which allows for linearization of the non-convex \emph{DistFlow} equations from which LinDistFlow is derived~\cite{LG-NL-UT-SL:2012}.

	\subsubsection*{Related Work} For a review of marginal demand costs in traditional OPF models including DC-OPF, the reader may consult the textbook~\cite{DSK-GS:2018}. Khatami et al. provide a detailed description of various components constituting nodal prices~\cite{RK-SN-YCC:2022}. Biegel et al. consider congestion management through shadow prices~\cite{BIEGEL2012518}. Bai et al. consider marginal pricing of real and reactive power demand under various markets for the nonlinear DistFlow model~\cite{LB-etal:2018}. Xu et al. design a deregulated power market mechanism, which uses the idea of marginal pricing at its core~\cite{7297853}. A comprehensive review of pricing mechanisms in transmission and distribution markets, including reserves, may be found in~\cite{MC-etal:2016}. Finally, similar in nature to the current paper, Winnicki et al. consider marginal pricing in the DistFlow model, but without consideration of congestion.~\cite{AW-MN-SB:2020}
	
	\subsubsection*{Contribution} In this paper, we formulate an OPF problem with the LinDistFlow model, named as LDF-OPF. We consider load satisfaction, generation bounds, voltage bounds, and conic branch flows. Our proposed framework can handle more generalized formulations with linear and conic constraints. We first show a closed-form expression for the marginal price when there is no flow congestion at optimality of LDF-OPF. Then, we derive an upper bound on the variation of demand marginal costs. The proposed upper bound is function of terms involving marginal costs associated with binding of aforementioned conic constrains representing branch flow limits, and network topology factors. This result provides a useful tool for network operators to estimate the change in demand marginal prices as a function of their choice of branch flow marginal costs. 
	
	\subsubsection*{Notation} $\mathbb{R}$ and $\mathbb{N}$ denote the set of real numbers and integers, respectively. Vectors and matrices are denoted with boldface. For a vector $\mb{a}\in\mathbb{R}^n$, $\mb{a}(j)$ is its $j^{\text{th}}$ element while $\norm{\mb{a}}_2$ denotes its 2-norm. $\onevec{n}\in\mathbb{R}^n$ is the all-ones vector. For a positive $n\in\mathbb{N}$, $[n]$ denotes the set $\{1,\cdots,n\}$. For a finite set $\mathcal{S}$, $|\mathcal{S}|$ denotes its cardinality. For a directed graph $\mcal{G}=(\mcal{N},\mcal{L})$ where $\mcal{N}$ is the set of nodes and $\mcal{L}$ the set of \emph{directed} edges, and for any node $i\in\mcal{N}$, $\mcal{H}_i$ denotes the \emph{inclusive downstream set} of $i$, i.e. $\mcal{H}_i \ldef \{j\in\mcal{N}|\exists \text{ directed path from $i$ to $j$ in }\mcal{L}\}\cup \{i\}$.
	
\section{Problem Formulation}
\subsection{Background}
Consider a radial power network with a slack bus. The buses are labeled as $\mcal{N}\ldef \{0,1,\cdots,n\}$, with $0$ denoting the slack bus. The network topology is represented by a directed graph $\mcal{G} \ldef (\mcal{N},\mcal{L})$, where $\mcal{L}$ denotes the set of branches. Without loss of generality, $\mcal{L}$ can be constructed such that the directed branches point \emph{away} from the slack bus. All non-slack buses are classified as the set of \emph{load} buses $\mcal{N}^l$ and the set of \emph{generator} buses $\mcal{N}^g$ such that $\mcal{N} = \{0\}\cup \mcal{N}^g \cup \mcal{N}^l$. Let $n_g \ldef |\mcal{N}^g|$ and $n_l\ldef |\mcal{N}^l|$ denote the number of generator and load buses, respectively. Noting that the number of branches equals that of non-slack buses, each branch may be uniquely assigned the index of the bus to which it is upstream.
	
A concrete way of analyzing the physics of power flows is through the \emph{DistFlow} equations. 
 For $i\in\mcal{N}$, let $s_i \ldef p_i + \mathfrak{i}q_i$ be the complex power injection at bus $i$, $S_i \ldef P_i + \mathfrak{i}Q_i$ be the complex power flow on branch $i$, and $v_i$ and $l_i$ denote the squared voltage magnitude at bus $i$, and squared current magnitude of branch $i$, respectively. The DistFlow equations that hold for all $i\in\mcal{N}$ are given as \cite{VK-LZ-GBG-RB:2016} 	
	\begin{subequations}
		\begin{align}
			s_i &= \sum_{j\in\text{child}(i)}S_j -S_i +l_iz_i,\\
			v_i &= v_{\text{parent}(i)} - 2\text{Re} [z_i^* S_i] + l_i |z_i|^2,\\
			|S_i|^2 &= v_{\text{parent}(i)}l_i.
		\end{align}  
		\label{model:DistFlow}
	\end{subequations}
	Assuming no power line losses in the network the DistFlow equations \ref{model:DistFlow} can be linearized into the so-called \emph{LinDistFlow} model whose compact form is given as \cite{VK-LZ-GBG-RB:2016}
	\begin{align}
		\label{model:LinDistFlow}
		\mb{v} = \mb{Rp} + \mb{Xq} + v_0 \onevec{n},
	\end{align}
	where $\mb{p} \ldef[\mb{p}(1),\cdots,\mb{p}(n)]$, $\mb{q}\ldef [\mb{q}(1),\cdots,\mb{q}(n)]$, and $\mb{v}\ldef [\mb{v}(1),\cdots, \mb{v}(n)]$ denote the real and reactive power injections and squared voltage magnitude of all non-slack buses, respectively. $v_0$ is the fixed voltage of the slack bus. Positive semi-definite matrices $\mb{R},\mb{X}\in\real{n\times n}$ encode branch resistance and reactance, as well as the topology of $\mcal{G}$. Since the vector $\mb{p}$ (similarly $\mb{q}$ and $\mb{v}$) contains various indices corresponding to generators and loads, we define the matrices $\mb{A}_g\in\{0,1\}^{n_g\times n}$ and $\mb{A}_l\in\{0,1\}^{n_l\times n}$ which help us separate generation and load indices from $p$ as
	\begin{align*}
		\mb{p}_g = \mb{A}_g \mb{p},\quad \mb{q}_g = \mb{A}_g \mb{q},\quad \mb{p}_l = \mb{A}_l \mb{p},\quad \mb{q}_l = \mb{A}_l\mb{q}.
	\end{align*}
	Thanks to zero line losses, the slack bus real and reactive injections become
	$p_s = -\onevec{n}^\top \mb{p}$ and $q_s = -\onevec{n}^\top \mb{q}$, respectively.
	
	 Let $\mb{f}^p\ldef [\mb{f}^\mb{p}(1),\cdots,\mb{f}^p(n)]$ and $\mb{f}^q\ldef [\mb{f}^\mb{p}(1),\cdots,\mb{f}^q(n)]$. The branch flows are given as
	$\mb{f}^p = \mb{Fp}$ and $\mb{f}^q = \mb{Fq}$, where $\mb{F}$ is derived from the \emph{signed branch-bus incidence matrix} $\tilde{\mb{A}}\in\real{ n\times n+1}$ by deleting its first column, and invert-transposing it.
	\begin{lemma}[Properties of matrix $\tilde{\mb{A}}$]
		\label{th:matf}
		The matrix $\tilde{\mb{A}}$ is defined as
		\begin{align*}
			\tilde{\mb{A}}(i,j) = \begin{cases}
				1,&\text{if branch $i$ starts at bus $j-1$}\\
				-1,&\text{if branch $i$ terminates at bus $j-1$}\\
				0,&\text{otherwise.}
			\end{cases}
		\end{align*}
		Let $\mb{A}$ be the square matrix derived by deleting the first column of $\tilde{\mb{A}}$. $\mb{A}^{-1}$ exists~\cite{VK-LZ-GBG-RB:2016}, and $\mb{F} \ldef \mb{A}^{-\top}$.
	\end{lemma}
	
	\subsection{Optimal Power Flow Problem} 
	A general OPF problem based on the described power network characteristics is given as follows.
	\begin{align}
		\label{eq:opf_orig}	\min\limits_{p,q,v,p_s,q_s} \quad&\mb{c}_g^\top (\mb{A}_g\mb{p}) + c_sp_s\\
		\tag{\theequation a}\label{cons:a}
		\text{s.t.} \quad& \mb{v} = \mb{Rp}+\mb{Xq}+v_0 \onevec{n}\\
		\tag{\theequation b}\label{cons:b}
		& \underline{\mb{v}} \leq \mb{v} \leq \bar{\mb{v}}\\
		\tag{\theequation c}\label{cons:c}
		& \mb{A}_l \mb{p} = \hat{\mb{p}}_l, \quad \mb{A}_l \mb{q} = \hat{\mb{q}}_l\\
		\tag{\theequation d}\label{cons:d}
		& \underline{\mb{p}}_g \leq \mb{A}_g \mb{p} \leq \bar{\mb{p}}_g, \quad \underline{\mb{q}}_g \leq \mb{A}_g \mb{q} \leq \bar{\mb{q}}_g\\
		\tag{\theequation e}\label{cons:e}
		& p_s = -\onevec{n}^\top \mb{p},\quad q_s = -\onevec{n}^\top \mb{q}\\
		\tag{\theequation f}\label{cons:f}
		& \underline{p}_s \leq p_s \leq \bar{p}_s,\quad \underline{q}_s \leq q_s \leq \bar{q}_s\\
		\tag{\theequation g}\label{cons:g}
		& \mb{f}^p = \mb{F p},\quad \mb{f}^q = \mb{F q}\\
		\tag{\theequation h}\label{cons:h}
		&\norm{ [\mb{f}^p(i),\mb{f}^q(i)]^\top}_2 \leq \bar{f}_i,\quad \forall i\in[n]
	\end{align}
	The objective function calculates the cost of generation at generator buses (given by $\mb{c}_g^\top(\mb{A}_g\mb{p})=\mb{c}_g^\top \mb{p}_g$, where $\mb{c}_g\in\real{n_g}$ is the generation cost vector) and at the slack bus (given by $c_s p_s$, where $c_s$ is the slack generation cost). ~\eqref{cons:a} is the LinDistFlow equation, while ~\eqref{cons:b} limits voltage values to within operational limits. ~\eqref{cons:c} stipulates that the demanded real and reactive power amounts are $\hat{\mb{p}}_l$ and $\hat{\mb{q}}_l$ respectively.~\eqref{cons:d},~\eqref{cons:e}, and~\eqref{cons:f} place limits on generation. Finally,~\eqref{cons:g} and~\eqref{cons:h} together describe the conic line flow for each branch.
	
	The OPF~\eqref{eq:opf_orig} contains redundancy in its decision variables and constraints. To that end, a much simplified and operatorized version of~\eqref{eq:opf_orig} may be written as follows.
	\begin{align}
		\label{eq:reduced_opf}
		\mcal{J}(\boldsymbol{\ell},\bar{\mb{f}}) \ldef \min\limits_{\mb{p}_g,\mb{q}_g} \quad&\tilde{\mb{c}}_g^\top \mb{p}_g\\
		\tag{\theequation a}\label{cons:mod_a}
		& \mb{M}_p \mb{p}_g + \mb{M}_q \mb{q}_g \leq \mb{G}\boldsymbol{\ell} + \mb{h}\\
		\tag{\theequation b}\label{cons:mod_b}
		& \underline{\mb{p}}_g \leq \mb{p}_g \leq \bar{\mb{p}}_g, \quad \underline{\mb{q}}_g \leq \mb{q}_g \leq \bar{\mb{q}}_g\\
		\tag{\theequation c}\label{cons:mod_c}
		&\norm{ \bm{ \mb{r}_i^\top \mb{p}_g + \mb{s}_i^\top \boldsymbol{\ell} \\ \mb{r}_i^\top \mb{q}_g + \mb{t}_i^\top \boldsymbol{\ell}} }_2 \leq \bar{f}_i, \quad \forall i \in [n]
	\end{align}
	In the above, $\tilde{\mb{c}}_g \ldef \mb{c}_g - c_s\onevec{n_g}$, $\boldsymbol{\ell}\ldef[\hat{\mb{p}}_l^\top,\hat{\mb{q}}_l^\top]^\top$, and $\bar{\mb{f}} \ldef [\bar{f}_1,\cdots,\bar{f}_n]^\top$. The OPF operator $\mcal{J}(\boldsymbol{\ell},\bar{\mb{f}})$ maps real \& reactive power demands and line flow limits to the optimum generation cost, which is a scalar. It can be shown that $\mcal{J}(\boldsymbol{\ell},\bar{\mb{f}})$ is jointly convex in $\boldsymbol{\ell}$ and $\bar{\mb{f}}$; see~\cite{citekeya}.
	\begin{remark}[Equivalence of~\eqref{eq:opf_orig} and~\eqref{eq:reduced_opf}]
		To establish the equivalence of~\eqref{eq:opf_orig} and~\eqref{eq:reduced_opf}, first $\mb{v}$, $p_s$, and $q_s$ can be eliminated as decision variables from~\eqref{eq:opf_orig} by replacing all of their occurrences with $\mb{Rp} + \mb{Xq} + v_0\onevec{n}$, $-\onevec{n}^\top \mb{p}$, and $-\onevec{n}^\top \mb{q}$ respectively. Further, the indices of $\mb{p}$ and $\mb{q}$ corresponding to load buses may be written as linear combinations of elements of $\boldsymbol{\ell}$ using~\eqref{cons:c}. Thus, the only effective decision variables in~\eqref{eq:opf_orig} are $\mb{p}_g, \mb{q}_g \in \real{n_g}$. Further, any expression in~\eqref{eq:opf_orig} which contains any of the decision variables $\mb{v,p,q},p_s$, or $q_s$ is linear in $\mb{p}_g$, $\mb{q}_g$ or $\boldsymbol{\ell}$. Therefore, the objective of~\eqref{eq:opf_orig} is equivalent to the objective of~\eqref{eq:reduced_opf}. Constraints~\eqref{cons:a}-\eqref{cons:e} are equivalent to ~\eqref{cons:mod_a},~\eqref{cons:f} and~\eqref{cons:mod_b}, and finally \eqref{cons:g}-\eqref{cons:h} are equivalent to ~\eqref{cons:mod_c}.
	\end{remark}
	\subsection{Illustrative Example}
	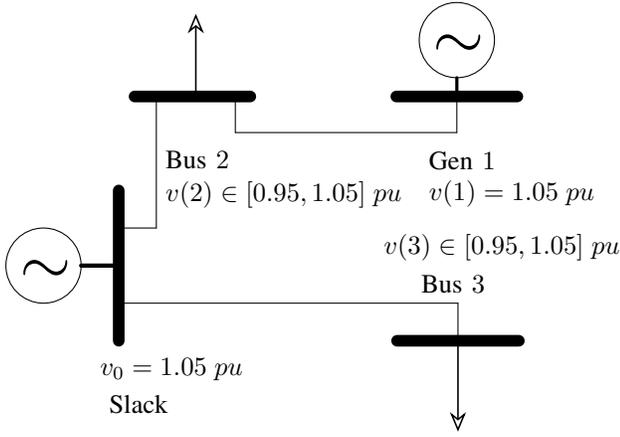
\begin{figure}[tb]
		\centering
		\begin{tikzpicture}[line cap=round,line join=round,x=1.0cm,y=1.0cm]
			\draw [line width=1.6pt] (-0.5,1.75)-- (0.,1.75);
			\draw(-1.,1.75) circle (0.5cm);
			\draw (-1.44,1.95) node[anchor=north west] {\Huge $\mathbf{\sim}$};
			\draw [line width=4.4pt] (0.,2.75) -- (0.,0.75);
			\draw (0.,2.25)-- (0.5,2.25);
			\draw (0.5,2.25)-- (0.5,4.00);
			\draw (1.55,3.52)-- (1.55,4.);
			\draw (4.495,3.52)-- (4.495,4.);
			\draw (1.55,3.52)-- (4.495,3.52);
			\draw (0.,1.25)-- (4.51,1.25);
			\draw (4.51,1.25)-- (4.51,0.75);
			\draw [line width=4.4pt] (0.21,4.00)-- (1.75,4.00);
			\draw [line width=4.4pt] (3.69,4.000)-- (5.30,4.00);
			\draw(4.495,4.75) circle (0.5cm);
			\draw (4.05,4.95) node[anchor=north west] {\Huge $\mathbf{\sim}$};
			\draw [line width=4.4pt] (3.70,0.75)-- (5.32,0.75);
			\draw [line width=1.2pt] (4.495,4.)-- (4.495,4.25);
			\draw [-{Stealth[scale=1.3,angle'=45,open]},semithick] (1.025,4.00) -- (1.025,5.1);
			\draw [-{Stealth[scale=1.3,angle'=45,open]},semithick] (4.51, 0.75) -- (4.51,-0.46);
			\draw (-0.25,0.15) node[anchor=north west] {$\text{Slack}$};
			\draw (3.9,1.75) node[anchor=north west] {$\text{Bus } 3$};
			\draw (0.5,3.4) node[anchor=north west] {$\text{Bus } 2$};
			\draw (0.5,3) node[anchor=north west] {$v(2) \in [0.95,1.05]\;pu$};
			\draw (4,3.4) node[anchor=north west] {$\text{Gen } 1$};
			\draw (-0.37,0.65) node[anchor=north west] {$v_0 = 1.05\;pu$};
			\draw (4,3.01) node[anchor=north west] {$v(1) = 1.05\;pu$};
			\draw (3.4,2.3) node[anchor=north west] {$v(3) \in [0.95,1.05]\;pu$};
		\end{tikzpicture}
		\caption{Illustrative example of a 4-bus radial network.}
		\label{fig:4bus}
  \vspace{-0.2cm}
	\end{figure}
	
	Consider a 4-bus system shown in Figure~\ref{fig:4bus}, which is derived from \texttt{case4ba} in MATPOWER~\cite{RDZ-etal:2011}. The buses are indexed such that the slack bus has index 0, the generator bus has index 1, and the two load buses have indices 2 and 3, respectively. Each branch is uniquely numbered based on its downstream bus. The impedance (in p.u.) of each branch is
	$0.003+j0.006$\,$\Omega$. 
	Thus, equation \eqref{cons:a} then becomes
	\begin{align*}
		\bm{\mb{v}(1)\\\mb{v}(2)\\\mb{v}(3)} &= \bm{0.012 & 0.006 & 	0\\0.006 & 0.006 & 0\\0 & 0 & 0.006}\bm{\mb{p}(1)\\\mb{p}(2)\\\mb{p}(3)} \\
		&+ \bm{0.024 & 0.012 & 0\\0.012 & 0.012 & 0\\0 & 0 & 0.012}\bm{\mb{q}(1)\\\mb{q}(2)\\\mb{q}(3)} + \bm{v_0 \\ v_0 \\ v_0}.
	\end{align*}
	The voltage of the generator bus and slack bus is fixed at $1.05$ p.u., while the voltage levels of the load buses vary in $[0.95,1.05]$ p.u.. Recognizing that $\boldsymbol{\ell}=[\mb{p}(2),\mb{p}(3),\mb{q}(2),\mb{q}(3)]$, $p_g = \mb{p}(1)$ and $q_g = \mb{q}(1)$, the above equation may be written as
	\begin{gather*}
		\mb{v}(1) = 0.012p_g + 0.024q_g + \bm{0.006&0&0.012&0}\boldsymbol{\ell} + v_0\\
		\bm{\mb{v}(2)\\\mb{v}(3)} = \bm{0.006\\0}p_g+ \bm{0.012\\0}q_g + \\
		\bm{0.006& 0& 0.012 & 0\\0 & 0.006 & 0 & 0.012}\boldsymbol{\ell} + \bm{v_0\\v_0}
	\end{gather*}
	Letting $\mb{v}(1) = 1.05$ (or equivalently, $1.05\leq \mb{v}(1)\leq 1.05$) and $0.95 \leq \mb{v}(i) \leq 1.05$ for $i=2,3$ recovers \eqref{cons:mod_a}. Now suppose the slack bus provides a maximum of 1 p.u. real power to the system; i.e., $p_s \leq 1$. This can be equivalently written in the form of \eqref{cons:mod_a} as $-\mb{p}_g - \bm{1&1&0&0} \boldsymbol{\ell} \leq 1$
Finally, we demonstrate a branch flow. The matrix $\mb{F}$ is given as
	\begin{align*}
		\mb{F} = \bm{-1 & 0 & 0\\-1 & -1 & 0\\0 & 0 & -1},
	\end{align*}
	and correspondingly
	\begin{gather*}
		\mb{f}^\mb{p}(1) = -\mb{p}(1), \quad \mb{f}^\mb{p}(2) = -\mb{p}(1)-\mb{p}(2), \quad \mb{f}^\mb{p}(3)=-\mb{p}(3)\\
		\mb{f}^\mb{q}(1) = -\mb{q}(1), \quad \mb{f}^\mb{q}(2) = -\mb{q}(1)-\mb{q}(2), \quad \mb{f}^\mb{q}(3)=-\mb{q}(3).
	\end{gather*}
	Suppose the branch flow limit of branch 2 is 3 p.u. It can be expressed as $\norm{\bm{\mb{f}^\mb{p}(2)&\mb{f}^\mb{q}(2)}^\top}_2 \leq 3$, equivalently,
	\begin{align*}
		\norm{\bm{(-p_g + \bm{-1 & 0&0&0}\boldsymbol{\ell}) \\ (-q_g + \bm{0 & 0&-1&0}\boldsymbol{\ell})}}_2 \leq 3.
	\end{align*}
	In other words, in the form of ~\eqref{cons:mod_c} we have
	\begin{align*}
		\mb{r}_2 &= [-1],
		\mb{s}_2^\top &= \bm{ -1 & 0 & 0 & 0},
		\mb{t}_2^\top &= \bm{0 & 0 & -1 & 0}.
	\end{align*}
For a generalized $j$, we provide the following closed form expressions on the norm values of $\mb{r}_j$, $\mb{s}_j$ and $\mb{t}_j$.
	\begin{lemma}[Norm values of $\mb{r}_j$, $\mb{s}_j$ and $\mb{t}_j$]
		\label{lem:sjtj}
		We have, for all branches $j$,
		\begin{gather*}
			\norm{\mb{r}_j}_2 =\sqrt{|\mcal{H}_j\cap \mcal{N}_g|},
			\norm{\mb{s}_j}_2 = \sqrt{|\mcal{H}_j\cap \mcal{N}_l|},\\
			\norm{\mb{t}_j}_2 = \sqrt{|\mcal{H}_j\cap \mcal{N}_l|}.
		\end{gather*}
	\end{lemma}
	\begin{proof}
		Recall that branch $j$ shares its index with its downstream bus $j$. We use the properties of matrix $\mb{F}$ in this proof. From Lemma~\ref{lem:sd}, $\mb{F}=\mb{A}^{-\top}$. It is known~\cite{DD-SB-MC:2018} that the $\tth{(i,j)}$ element of matrix $\mb{A}^{-1}$ is +1 if branch $j$ is directed \emph{along} path from $i$ to slack bus, -1 if it is directed \emph{against}, and 0 otherwise. By the construction of $\mcal{G}$, $\mb{A}^{-1}$ only has entries -1 and 0. Thus, $\mb{A}^{-\top}$ has $\tth{(i,j)}$ element -1 if branch $i$ is falls in the path from bus $j$ to the origin and 0 otherwise. Therefore, the $\tth{i}$ row of $\mb{A}^{-\top}$ collects all buses in $\mcal{H}_i$.
		 The result follows by separating generator buses and placing their coefficients in $\mb{r}_j$, and load buses in $\mb{s}_j$ and $\mb{t}_j$.
	\end{proof}
	\section{Marginal Analysis}
	The analysis of marginal pricing starts from the dual problem of~\eqref{eq:reduced_opf} that is given as follows.
	\begin{lemma}
		\label{lem:dual}
		The dual problem of~\eqref{eq:reduced_opf} is given as
		\begin{align}
			\label{eq:dual}
			\max\limits_{\substack{\pmb{\lambda},\pmb{\alpha}_{lb},\pmb{\alpha}_{ub},\pmb{\beta}_{lb},\\\pmb{\beta}_{ub},\theta_i,\phi_i,\mu_i}} \quad& -(\mb{G}\boldsymbol{\ell}+\mb{h})^\top \pmb{\lambda} + \pmb{\alpha}_{lb}^\top \underline{\mb{p}}_g - \pmb{\alpha}_{ub}^\top\bar{\mb{p}}_g + \pmb{\beta}_{lb}^\top \underline{\mb{q}}_g \notag \\
			& -\pmb{\beta}_{ub}^\top \bar{\mb{q}}_g + \sum_{i=1}^n \left[-(\theta_i \mb{s}_i + \phi_i \mb{t}_i)^\top \boldsymbol{\ell} - \mu_i \bar{f}_i \right] \\
			\tag{\theequation a}\label{eqc1}
			\text{s.t.} \quad \tilde{\mb{c}}_g + &\mb{M}_p^\top \pmb{\lambda} + \pmb{\alpha}_{ub}-\pmb{\alpha}_{lb} - \sum_{i=1}^n \theta_i \mb{r}_i = 0\\
			\tag{\theequation b}\label{eqc2}
			\mb{M}_q^\top \pmb{\lambda} &+ \pmb{\beta}_{ub}-\pmb{\beta}_{lb} - \sum_{i=1}^n \phi_i \mb{t}_i = 0\\
			\tag{\theequation c}\label{eqc3}
			&\norm{[\theta_i,\phi_i]^\top}_2 \leq \mu_i,\quad \forall i\in[n]\\
			\tag{\theequation d}\label{eqc4}
			& \pmb{\lambda},\pmb{\alpha}_{lb},\pmb{\alpha}_{ub},\pmb{\beta}_{lb},\pmb{\beta}_{ub},\{\mu_i\} \geq 0
		\end{align}
	\end{lemma}
	\begin{proof}
		We augment problem~\eqref{eq:reduced_opf} by adding auxiliary variables $y_i\ldef\mb{r}_i^\top \mb{p}_g+ \mb{s}_i^\top \pmb{\ell}$ and $z_i\ldef \mb{r}_i^\top\mb{q}_g + \mb{t}_i^\top \pmb{\ell}$, and letting $\theta_i$ and $\phi_i$ denote the dual variables for the same. Constraint~\eqref{eqc3} can then be written as $\norm{[y_i,z_i]^\top}_2 \leq \bar{f}_i,\forall i \in [n]$.
		The Lagrangian of the augmented problem is given as
		\begin{align*}
			\mcal{L} \ldef& \tilde{\mb{c}}_g^\top \mb{p}_g + \pmb{\lambda}^\top (\mb{M}_p \mb{p}_g + \mb{M}_q \mb{q}_g - \mb{G}\boldsymbol{\ell} - \mb{h}) + \pmb{\alpha}_{lb}^\top (\underline{\mb{p}}_g - \mb{p}_g) \\
			& + \pmb{\alpha}_{ub}^\top (\mb{p}_g - \bar{\mb{p}}_g) + \pmb{\beta}_{lb}^\top (\underline{\mb{q}}_g - \mb{q}_g)  +\pmb{\beta}_{ub}^\top (\mb{q}_g - \bar{\mb{q}}_g) +\\
			& +\sum_{i=1}^n \bigg[ \theta_i(y_i-\mb{r}_i^\top \mb{p}_g - \mb{s}_i^\top \pmb{\ell}) + \phi_i(z_i-\mb{r}_i^\top \mb{q}_g - \mb{t}_i^\top \pmb{\ell}) \\
			& + \mu_i\left(\norm{y_i,z_i]^\top}_2 - \bar{f}_i\right) \bigg].
		\end{align*}
		The dual objective function consists of all the terms in $\mcal{L}$ which are not functions of any primal variables. Since $\mcal{L}$ is linear in $\mb{p}_g$ and $\mb{q}_g$, their respective coefficients must be zero such that $\mcal{L}$ is bounded from below. This gives rise to \eqref{eqc1} and \eqref{eqc2}. Finally, that for any $y,z,\theta,\phi\in\real{}$ and $\mu > 0$, it holds that
		\begin{align*}
			\inf\limits_{y,z} \;\;[\theta,\phi]\bm{y\\z} + \mu \norm{\bm{y\\z}}_2 = \begin{cases}
				0, & \text{if } \norm{[\theta,\phi]^\top}_2 \leq \mu,\\
				-\infty, & \text{if } \norm{[\theta,\phi]^\top}_2 > \mu.
			\end{cases}
		\end{align*}
		Applying this property to find the infimum of all terms consisting of $y_i$ and $z_i$ in $\mcal{L}$, we recover~\eqref{eqc3}. Finally, non-negativity of dual variables yields ~\eqref{eqc4}.
	\end{proof}
	We assume that there exist optimal primal and dual solutions for \eqref{eq:reduced_opf} and \eqref{eq:dual} such that Kahrush-Kuhn-Tucker (KKT) conditions hold \cite[Section 5.5.3]{boyd2004convex}.
\begin{assmp}[KKT conditions]
\label{lem:sd}
Let  $\Gamma \ldef \{ \pmb{\lambda},\pmb{\alpha}_{lb},\pmb{\alpha}_{ub},\pmb{\beta}_{lb},\pmb{\beta}_{ub},\theta_i,\phi_i,\mu_i\}$ denote the set of dual variables in \eqref{eq:dual}. Then, there exist optimal solutions $\mb{p}_g^*$ and $\mb{q}_g^*$ for primal problem~\eqref{eq:reduced_opf} and $\Gamma^*$ for dual problem~\eqref{eq:dual} which satisfy the following KKT conditions:
\begin{enumerate}
\item stationarity of $\mcal{L}$ in primal variables $\mb{p}_g$ and $\mb{q}_g$.
\item constraints \eqref{cons:mod_a}--\eqref{cons:mod_c} hold for primal variables.
\item constraints \eqref{eqc1}--\eqref{eqc4} hold for dual variables.
\end{enumerate}
\begin{align*}
\pmb{\lambda}^\top (\mb{M}_p \mb{p}_g + \mb{M}_q \mb{q}_g - \mb{G}\boldsymbol{\ell} - \mb{h}) = 0\\	\pmb{\alpha}_{lb}^\top(\underline{\mb{p}}_g - \mb{p}_g)=0, \quad \pmb{\alpha}_{ub}^\top(\mb{p}_g- \bar{\mb{p}}_g)=0\\	\pmb{\beta}_{lb}^\top(\underline{\mb{q}}_g - \mb{q}_g)=0, \quad \pmb{\beta}_{ub}^\top({\mb{q}}_g- \bar{\mb{q}}_g)=0.
\end{align*}
\end{assmp}

We now introduce the \emph{dual value function} $\mcal{D}(\Gamma,\boldsymbol{\ell},\bar{\mb{f}})$, which is defined as the objective function of dual problem~\eqref{eq:dual} as a function of \emph{any} values of dual variables $\Gamma$, demands $\boldsymbol{\ell}$, and flow limits $\bar{\mb{f}}$. Assumption~\ref{lem:sd} allows us to exploit duality theory in order to equate the dual value function to the operator $\mcal{J}(\boldsymbol{\ell},\bar{\mb{f}})$.
	\begin{remark}[Strong duality]
		\label{rem:sd}
		Since the primal problem~\eqref{eq:reduced_opf} is convex in the decision variables, Assumption~\ref{lem:sd} ensures that strong duality holds~\cite[Section 5.5.3]{boyd2004convex}, i.e. $\mcal{J}(\boldsymbol{\ell},\bar{\mb{f}}) = \mcal{D}(\Gamma^*,\boldsymbol{\ell},\bar{\mb{f}})$.
	\end{remark}
	We now provide a closed form of the \emph{flow marginal costs} as a function of the optimum dual variables.
	\begin{lemma}
		\label{lem:fmc}
		The flow marginal cost for line flow limits $\bar{\mb{f}}$, denoted as $C^{flow}_{\bar{\mb{f}}}$ is given as
		\begin{align}
			C^{\text{flow}}_{\bar{\mb{f}}}(j) \ldef \nabla_{\bar{\mb{f}}(j)} \mcal{J}(\boldsymbol{\ell},\bar{\mb{f}}) = -\mu^*_j,\quad \forall j\in[n]
		\end{align}
		Moreover, if $\mb{p}_{g1}^*$ is the optimum real generation with flow limits $\bar{\mb{f}}_1$ and $\mb{p}_{g2}^*$ with $\bar{\mb{f}}_2$ where $\bar{\mb{f}}_1 \succeq \bar{\mb{f}}_2$, then $\tilde{\mb{c}_g}^\top \mb{p}_{g1}^* \leq \tilde{\mb{c}_g}^\top \mb{p}_{g2}^*$.
	\end{lemma}
	\begin{proof}
		Due to strong duality, it follows that $\mcal{J}(\boldsymbol{\ell},\bar{\mb{f}}) = \mcal{D}(\Gamma^*,\boldsymbol{\ell},\bar{\mb{f}})$. Taking derivative of $\mcal{D}(\Gamma^*,\boldsymbol{\ell},\bar{\mb{f}})$ with respect to each of the elements in $\bar{\mb{f}}$ recovers the closed form of $C^{flow}_{\bar{\mb{f}}}$. The second part of the result follows from observing that the feasible space of~\eqref{eq:reduced_opf} is smaller under $\bar{\mb{f}}_2$ than $\bar{\mb{f}}_1$, leading to same or higher objective value.
	\end{proof}
	To this end, we present the main result as follows.
	\begin{thm}[Bounding marginal prices]
		\label{prop:1}
		Suppose at optimality of problem~\eqref{eq:reduced_opf}, $\mcal{I}(\bar{\mb{f}}_1)\subseteq[n]$ is the nonempty index set collecting the binding constraints in \eqref{cons:mod_c} for flow limits $\bar{\mb{f}}_1$. Then, the congested marginal cost of load demand is given as 
		\begin{align}
			C^{\text{load}}_{\bar{\mb{f}}_1} \ldef \nabla_{\boldsymbol{\ell}} \mcal{J}(\boldsymbol{\ell},\bar{\mb{f}}_1) = -\pmb{\lambda}^{*\top}\mb{G}  - \sum_{i\in\mcal{I}(\bar{\mb{f}}_1)} (\theta_i^* \mb{s}_i + \phi_i^* \mb{t}_i)^\top .
		\end{align}
\begin{figure*}[ht!]
	\includegraphics[width=\linewidth]{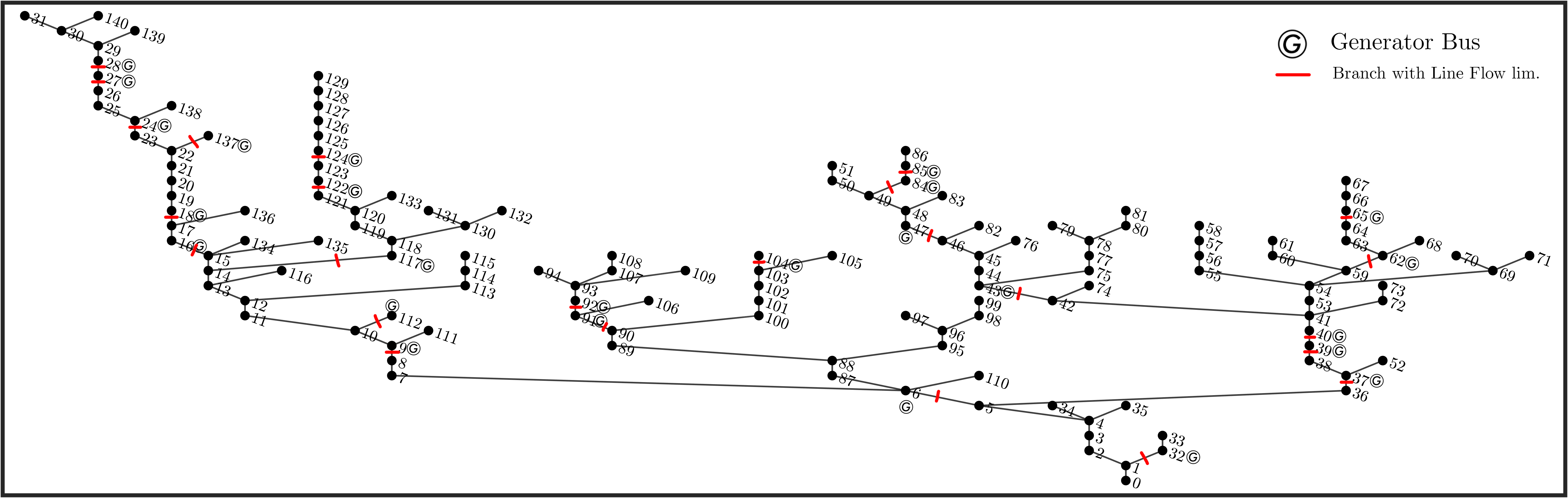}
	\caption{141-bus distribution system, adapted from MATPOWER \texttt{case141}. We retain the topology and real/reactive load demands from \texttt{case141}, but randomly add 25 distributed generation buses. The real power generation of each distributed generator is limited to $p_g\in[0,0.0654]pu$, while the reactive power generation follows limits $q_g \in [-0.0270,0.0270]pu$.}
	\label{fig:topology}
\end{figure*}
\begin{figure*}[ht!]
	\centering
	\includegraphics[width=0.9\linewidth]{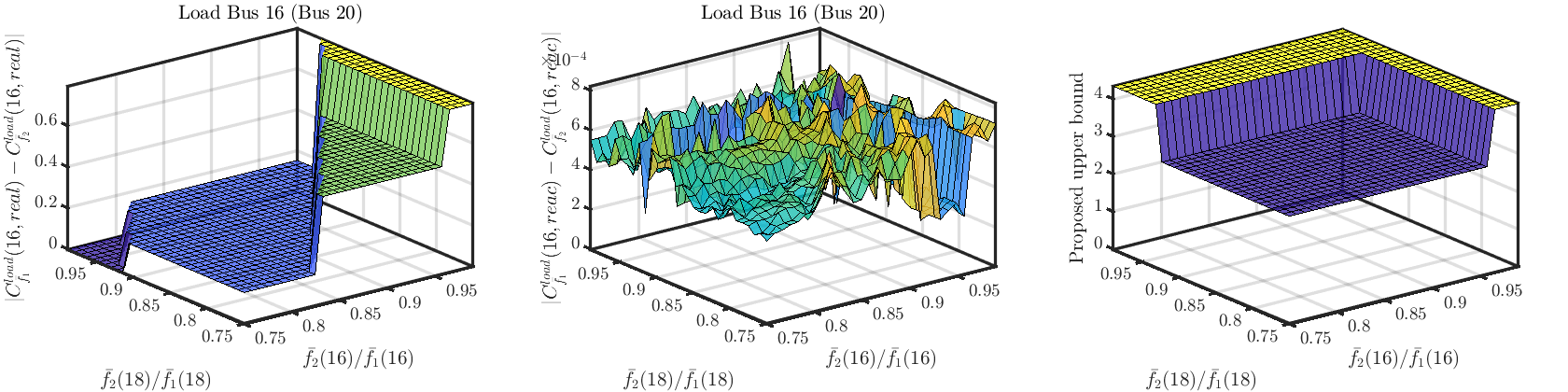}
	\caption{Results from simultaneously perturbing two branch flow limits from Figure~\ref{fig:topology}. We choose branches 16 and 18 to perturb since branch 18 is downstream to branch 16, making for easier interpretability of results. The first two figures indicate marginal costs of real and reactive power respectively, while the third figure is the proposed upper bound on the quantities derived in both the first and second figure. As we see, the upper bound is valid for both the quantities.}
	\label{fig:results}
\end{figure*}
		Furthermore, let $\mcal{I}(\bar{\mb{f}}_2)$ be empty for flow limits $\bar{\mb{f}}_2$, denoting the case wherein the network is uncongested. We have,
		\begin{gather}
			\tag{real}
			|C^{\text{load}}_{\bar{\mb{f}}_1}(i) - C^{\text{load}}_{\bar{\mb{f}}_2}(i)| \leq K_{\bar{\mb{f}}_1}\sum_{j\in \mcal{I}(\bar{\mb{f}}_1)} |C^{\text{flow}}_{\bar{\mb{f}}_1}(j)|, \\
			\tag{reactive}
			|C^{\text{load}}_{\bar{\mb{f}}_1}(n_l+i) - C^{\text{load}}_{\bar{\mb{f}}_2}(n_l+i)| \leq K_{\bar{\mb{f}}_1} \sum_{j\in \mcal{I}(\bar{\mb{f}}_1)}|C^{\text{flow}}_{\bar{\mb{f}}_1}(j)|
		\end{gather}
		for all $i\in[n]$, $K_{\bar{\mb{f}}_1}>0$ is a constant defined as $K_{\bar{\mb{f}}_1} \ldef {(n_{nz}^{(j)})}^{\frac{1}{2}}\max\limits_{j\in\mcal{I}(\bar{\mb{f}}_1)}~\norm{[\mb{s}_j^\top,\mb{t}_j^\top]^\top}_2$, where $n_{nz}^{(j)}$ is the number of nonzeros in $\mb{s}_j$ (same as $\mb{t}_j$).
	\end{thm}
	\begin{proof}
		The closed form of $C^{\text{load}}_{\bar{\mb{f}}_1}$ may be derived similar to the proof of Lemma~\ref{lem:fmc}. Note that
		$C^{\text{load}}_{\mb{\bar{f}}_1} - C^{\text{load}}_{\mb{\bar{f}}_2} = -\sum_{i\in\mcal{I}(\bar{\mb{f}}_1)} (\theta_i^* \mb{s}_i + \phi_i^* \mb{t}_i)^\top$. Thus,
		\begin{gather*}
			\left|C^{\text{load}}_{\mb{\bar{f}}_1}(i) - C^{\text{load}}_{\mb{\bar{f}}_2}(i)\right| \\
			= \left| -\sum_{j\in \mcal{I}(\bar{\mb{f}}_1)} \theta^*_j \mb{s}_j + \phi^*_j \mb{t}_j\right| \leq \sum_{j\in \mcal{I}(\bar{\mb{f}}_1)} \left| \bm{\mb{s}_j\\\mb{t}_j}^\top\bm{\theta_j^* \mb{z}^s_j\\ \phi^*_j\mb{z}^t_j} \right|\\
			\stackrel{[a]}{\leq} \sum_{j\in \mcal{I}(\bar{\mb{f}}_1)}\left( \norm{\bm{\mb{s}_j\\\mb{t}_j}}_2\times \norm{\bm{\theta_j^*\mb{z}^s_{j}\\\phi^*_j\mb{z}^t_{j}}}_2 \right)\quad\stackrel{[b]}{\leq} \sum_{j\in\mcal{I}(\bar{\mb{f}}_1)} \mu_j^* K_{\bar{\mb{f}}_1},
		\end{gather*}
		where $\mb{z}^s_j$ and $\mb{z}^t_j$ are vectors which have value $1$ at indices where $\mb{s}_j$ and $\mb{t}_j$ are nonzero respectively, and $0$ otherwise. $[a]$ is due to the Cauchy-Schwartz inequality while $[b]$ follows from dual ~\eqref{eqc3} and uniformly upper bounding the term $\norm{[\mb{s}_j^\top,\mb{t}_j^\top]^\top}_2$ for all $j\in\mcal{I}(\bar{\mb{f}}_1)$.
		Observing that $\nabla_{\bar{f}_j} \mcal{J}(\boldsymbol{\ell},\bar{\mb{f}}) = -\mu_j^*$ and $\mu_j^*\geq 0$ concludes the proof. 
	\end{proof}
	Lemma~\eqref{lem:sjtj} may be used to find the closed form of $K_{\mb{\bar{f}}}$ for different $\bar{\mb{f}}$.
	In the following remark, we highlight some use cases of Theorem~\ref{prop:1}.
	\begin{remark}[Use cases of Theorem~\ref{prop:1}]
		Note that system operators often possess datasets of the form $\{ {C^{\text{flow}}_{\mb{\bar{f}}}}_k,\bar{\mb{f}}_k,p_k\}$, where $p_k\in[0,1]$ is the probability of sample $k$ being representative of a desired scenario. Such a probability may be derived empirically by observing frequencies of similar data points in the dataset
		\begin{itemize}
			\item The bounds may be used to quickly approximate the worst-case changes in power tariffs as a function of implemented line flow limits $\bar{\mb{f}}$ using existing datasets, without re-running the full OPF problem for different values of $\bar{\mb{f}}$. 
			\item The capability to compute the marginal prices $C^{\text{flow}}$ and $C^{\text{load}}$ as a function of $\Bell$ and $\bar{\mb{f}}$ may be used to train novel sensitivity informed deep learning architectures~\cite{MKS-VK-GBG:2022}. Bounds based on sensitivity of the OPF solution to input parameters can also be used to construct triggering conditions in distributed optimization using the principle of event-triggered communication~\cite{sb-pt:2021}.
			\item If the power demand $\boldsymbol{\ell}$ is a random variable drawn from a known distribution, then Theorem~\ref{prop:1} can provide a method to quantify the statistical properties of demand marginal costs. Such analyses generally fall under the domain of probabilistic (optimal) power flow, which is a topic of significant research due to its possible applications in deep-learning based architectures for power systems~\cite{9817505}.
		\end{itemize}
	\end{remark}
	\section{Simulation Result}

In this section, we experimentally validate the bounds derived in Theorem~\ref{th:matf} by carrying out simulations on a 141-bus single-phase radial power network derived from \texttt{case141} in MATPOWER~\cite{RDZ-etal:2011}, as shown in Figure~\ref{fig:topology}. Due to \texttt{case141} originally containing only a single generator at the slack bus, we modify the same to add distributed generation. This is done according to the following steps.
\begin{itemize}
	\item First, 25 buses are randomly selected to host distributed generation. Then per generator bounds are chosen such that the total upper bound of real power generation approximately equals the total real power demand at the load buses. The real power cost of distributed generator buses are chosen uniformly-at-random in the range $[0,1]$, while the slack bus generation is chosen to be more expensive than any distributed generator.
	\item Following introduction of distributed generators, we run the OPF without any branch flow limits to generate optimal generation and branch flow setpoints. This is done using CVX~\cite{cvx} and MATLAB. Once the optimal flow setpoints are generated, the branch flow limit of the line upstream to any distributed generator is set to its corresponding flow setpoint.
	\item Two branches, \emph{viz.} 16 and 18 are selected, and their branch flow limits are simultaneously reduced upto 75\% of their original values. For each step in the reduction, the perturbation of real and reactive marginal cost of load demand at bus 20 (a load bus) is recorded, and presented in Figure~\ref{fig:results}. The proposed upper bound is also presented in the same figure.
\end{itemize}

	As seen in Figure~\ref{fig:results}, the proposed bounds are respected by the actual perturbation in marginal costs. Here, both the marginal cost perturbations and the upper bound are derived from the formulations presented in Theorem~\ref{th:matf}. It should be noted that the results in Theorem~\ref{th:matf} use optimal dual variables in its formulation, which are generated in most modern optimization solvers. Another interesting observation in Figure~\ref{fig:results} is that the marginal cost of reactive power is within some numerical error of zero. This is because the objective function of OPF~\ref{eq:opf_orig} does not penalize generation/consumption of reactive power. Were an objective term concering reactive power to be added to the OPF, we would observe non-zero marginal costs for reactive power alongside real power demand.
	
	\section{Conclusion}
	In this paper, we formulated the LDF-OPF problem and derived upper bounds on the marginal prices of real and reactive power demand for all load buses. We also presented certain approaches showing how the main result in Theorem~\ref{prop:1} can be utilized for OPF-based planning problems. Future work will develop results in this paper to accelerate solve times of LDF-OPF.
	
	\section*{Acknowledgements}
	The authors would like to thank Professor Yihsu Chen at the University of California, Santa Cruz for helpful discussions on this work and its future directions.

\end{document}